\documentclass[smallextended]{svjour3}
\smartqed

\usepackage{fix-cm}
\usepackage{cite}
\usepackage{graphicx}
\usepackage{psfrag}
\usepackage{epsfig}
\usepackage{url}
\usepackage{stfloats}
\usepackage{amsfonts,amssymb,amsmath,bm,paralist,theorem,cite,ifthen,color,amsmath}
\usepackage{array}
\usepackage{caption}
\usepackage{calc}
\usepackage{longtable}
\usepackage{subfigure}
\usepackage{mathrsfs}
\usepackage{epsfig}
\usepackage{float}

\newtheorem{thm}{Theorem}

\theorembodyfont{\rmfamily}

\newtheorem{rmk}{Remark}
\newcommand{\tr}{{\rm Tr}}
\newcommand{\st}{{\rm s.t.}}
\newcommand{\re}{\mathbb{R}}
\newcommand{\bW}{{\bf W}}

\newcommand{\bZ}{{\bf Z}}

\newcommand{\cX}{{\cal X}}
\newcommand{\xbar}{{\bar x}}

\newcommand{\bV}{{\textbf{V}}}
\newcommand{\bH}{{\mathbf{H}}}

\newcommand{\ghat}{{\hat g}}
\newcommand{\fhat}{{\hat f}}
\newcommand{\bA}{{\textbf{A}}}
\newcommand{\bs}{\mathbf{s}}

\newcommand{\bx}{\mathbf{x}}
\newcommand{\by}{\mathbf{y}}
\newcommand{\bI}{\mathbf{I}}
\newcommand{\bn}{\mathbf{n}}
\newcommand{\bE}{\mathbf{E}}
\newcommand{\bP}{\mathbf{P}}
\newcommand{\bU}{\mathbf{U}}

\newtheorem{lem}{Lemma}
\title{A Stochastic Successive Minimization Method for Nonsmooth Nonconvex Optimization}
\subtitle{with Applications to Transceiver Design in Wireless Communication Networks}

\author{Meisam Razaviyayn \and Maziar Sanjabi \and Zhi-Quan Luo
}

\institute{Meisam Razaviyayn \and Maziar Sanjabi \and Zhi-Quan Luo \at
              Department of Electrical and Computer Engineering, University of Minnesota, 200 Union Street SE, Minneapolis, MN 55455. \\
              \email{\{meisam,maz,luozq\}@umn.edu}           
}
\date{Received: date / Accepted: date}
\begin{document}

\maketitle

\begin{abstract}
Consider the problem of minimizing the expected value of a cost function parameterized by a random variable. The classical sample average approximation (SAA) method for solving this problem requires minimization of an ensemble average of the objective at each step, which can be expensive. 
In this paper, we  propose a stochastic successive upper-bound minimization method (SSUM) which minimizes an \emph{approximate} ensemble average at each iteration. To ensure convergence and to facilitate computation, we require the approximate ensemble average to be a locally tight upper-bound of the expected cost function and be easily optimized. The main contributions of this work include the development and analysis of the SSUM method as well as its applications in linear transceiver design for wireless communication networks and online dictionary learning. Moreover, using the SSUM framework, we extend the classical stochastic (sub-)gradient (SG) method to the case of minimizing a nonsmooth nonconvex objective function and establish its convergence. 

\end{abstract}



\section{Introduction}
\label{sec:intro}
Consider the optimization problem
\begin{equation}
\label{eq:OriginalProblem}
\begin{split}
\min \quad &f(x) \triangleq \mathbb{E}_\xi [g(x,\xi)]\\
\st \quad &x \in \mathcal{X},
\end{split}
\end{equation}
where $\mathcal{X} \subseteq  \re^{n}$ is a bounded closed convex set; $\xi$ is a random vector drawn from a set $\Xi \in \re^m$, and $g:\cX \times \Xi \mapsto \mathbb{R}$ is a real-valued function.  A classical approach for solving the above optimization problem is the sample average approximation (SAA) method. At each iteration of the SAA method, a new realization  of the random vector $\xi$ is obtained and the optimization variable $x$ is updated by solving
\begin{equation}
\label{eq:SAA}
\begin{split}
x^r \in \; \arg \min \;\;&\frac{1}{r}\sum_{i=1}^r g(x,\xi^i)\\
\st \quad & x \in \mathcal{X}.
\end{split}
\end{equation}%
Here $\xi^1,\xi^2, \ldots$ are some independent, identically distributed realizations of the random vector $\xi$. 
 We refer the readers to \cite{plambeck1996sample,robinson1996analysis, healy1991retrospective, rubinstein1990optimization, rubinstein1993discrete} for the roots of the SAA method and \cite{shapiro2003monte, shapiro2009lectures, kim2011guide} for several surveys on SAA.

 A main drawback of the SAA method is the complexity of each step. In general, solving \eqref{eq:SAA} may not be easy due to the non-convexity and/or non-smoothness of $g(\cdot,\xi)$. To overcome the difficulties in  solving the subproblem \eqref{eq:SAA}, we propose an inexact SAA method whereby at each step a well-chosen approximation of the function $g(\cdot, \xi)$ in \eqref{eq:SAA} is minimized. Specifically, at each iteration $r$, we update the optimization variable according to
\begin{equation}
\label{eq:SSUMupdate}
x^r \leftarrow \arg\min_{x \in\mathcal{X}} \; \frac{1}{r} \sum_{i=1}^r \ghat(x,x^{i-1},\xi^i),
\end{equation}
where $\ghat(\cdot,x^{i-1},\xi^i)$ is an approximation of the function $g(\cdot,\xi^i)$ around the point $x^{i-1}$. To ensure the convergence of this method, we require the approximation function $\ghat(\cdot, x^{i-1}, \xi^i)$ to be a locally tight upper bound of the original function $g(\cdot,\xi^i)$ around the point $x^{i-1}$, for each $i=0,\ldots,r-1$. For this reason, we call the above algorithm \eqref{eq:SSUMupdate} a stochastic successive upper-bound minimization method (SSUM).

The idea of successive upper-bound minimization (also known as majorization minimization or successive convex optimization) has been widely studied in the literature for deterministic optimization problems; see \cite{BSUM} and the references therein. In the successive upper-bound minimization (SUM) framework, a locally tight approximation of the function is minimized at each step of the algorithm. This technique is key to many important practical algorithms such as the concave-convex procedure \cite{CCCP} and the expectation maximization algorithm \cite{EMTutorial,EMDempster}.

While the successive upper-bound minimization idea is well studied and widely used in deterministic settings, very little is known about its use in the stochastic setup. The main contributions of this paper are to extend the technique of successive upper-bound minimization to the stochastic setup \eqref{eq:OriginalProblem} and to illustrate its use in applications. In particular, we first establish the convergence of SSUM defined by \eqref{eq:SSUMupdate}, and then describe two important applications of the SSUM framework: the sum rate maximization problem for wireless communication networks and the online dictionary learning problem.  For the stochastic wireless beamforming problem, our numerical experiments indicate that the SSUM approach significantly outperforms the other existing algorithms in terms of the achievable ergodic sum rate in the network. In addition, we show that the traditional stochastic gradient (SG) algorithm for unconstrained smooth minimization is a special case of the SSUM method. Moreover, using the SSUM framework, we extend the SG algorithm to the problem of minimizing a nonsmooth nonconvex objective function and establish its convergence.

\subsection{Technical Preliminaries}
Throughout the paper we adopt the following notations. We use $\mathbb{R}$ to signify the set of real numbers. The notation $\mathbb{E}(\cdot)$ is used to represent the expectation operator. Unless stated otherwise, all relations between random variables hold almost surely in this paper. A list of other definitions adopted in the paper are given below.
\begin{itemize}
\item \textbf{Distance of a point to a set:} Given a non-empty set $S\subset \mathbb{R}^{n}$ and a point $x\in \mathbb{R}^{n}$, the distance of the point $x$ to the set $S$ is defined as
\begin{align}
d(x,S)~\triangleq~\inf_{s\in S}~\|x-s\|, \nonumber
\end{align}
where $\|\cdot\|$ denotes the 2-norm in $\mathbb{R}^{n}$.
\item \textbf{Directional derivative:} Let $h:~D\mapsto\mathbb{R}$ be a function, where $D\subseteq \mathbb{R}^{n}$ is a convex set. The directional derivative of the function $h$ at a point $x\in D$ in the direction $d\in R^{n}$ is defined as
\begin{align}
h'(x;d)~\triangleq ~\liminf_{t\downarrow 0}\; \frac{h(x+td)-h(x)}{t}. \nonumber
\end{align}
Moreover, we define $h'(x;d) \triangleq +\infty$, if $x+td\notin D$, $\forall\; t>0$.
\item \textbf{Stationary points of a function:} Let $h:~D\mapsto\mathbb{R}$ be a function, where $D\subseteq \mathbb{R}^{n}$ is a convex set. The point $x\in \mathbb{R}^{n}$ is a stationary point of $h(\cdot)$ if
\begin{align}
h'(x;d)\geq 0, \;\;\;\;\;\;\;\forall~d\in \mathbb{R}^{n}. \nonumber
\end{align}
\item \textbf{Natural history of a stochastic process:} Consider a real valued stochastic process $\{Z^{r}\}_{r=1}^{\infty}$. For each $r$, we define the natural history of the stochastic process up to time $r$ as
\begin{align}
\mathcal{F}^{r}~=~\sigma (Z^{1},\ldots,Z^{r}), \nonumber
\end{align}
where $\sigma(Z^{1},\ldots,Z^{r})$ denotes the $\sigma$-algebra generated by the random variables $Z^{1},\ldots,Z^{r}$.
\item \textbf{Infinity norm of a function:} Let $h:~D\mapsto \mathbb{R}$ be a function, where $D\subseteq \mathbb{R}^{n}$. The infinity norm of the function $h(\cdot)$ is defined as
\[
\|h\|_\infty \triangleq \sup_{x\in D} |h(x)|.
\]
\end{itemize}
\section{Stochastic Successive Upper-bound Minimization}
\label{sec: SSUM}
In the introduction section, we have briefly outlined the Stochastic Successive Upper-bound Minimization (SSUM) algorithm as an inexact version of the SAA method. In this section, we describe the SSUM algorithm and the associated assumptions more precisely, and provide a complete convergence analysis.

Consider the optimization problem
\begin{align}
\min~&~\left\{f(x)\triangleq \mathbb{E}_{\xi}\left[g_{1}(x,\xi)+g_{2}(x,\xi)\right]\right\}\label{eq: Prob}\\
\st ~&~ x\in \mathcal{X},\nonumber
\end{align}
where $\mathcal{X}$ is a bounded closed convex set and $\xi$ is a random vector drawn from a set $\Xi\in \mathbb{R}^{m}$. We assume that the function $g_{1}:~\mathcal{X}\times \Xi \mapsto \mathbb{R}$ is a continuously differentiable (and possibly non-convex) function in $x$, while  $g_{2}:~\mathcal{X}\times \Xi\mapsto \re$ is a convex continuous (and possibly non-smooth) function in $x$. Due to the non-convexity and non-smoothness of the objective function, it may be difficult to solve the subproblems \eqref{eq:SAA} in the SAA method. This motivates us to consider an inexact SAA method by using an approximation of the function $g(\cdot,\xi)$ in the SAA method \eqref{eq:SAA} as follows:
\begin{align}
x^{r}~\leftarrow~&\arg\min_{x}\;\;\frac{1}{r}\sum_{i=1}^{r}\left({\hat g}_{1}(x,x^{i-1},\xi^{i})+g_{2}(x,\xi^{i})\right)\label{eq: SSUM}\\
&\st ~~~~x\in \mathcal{X},\nonumber
\end{align}
where ${\hat g}_{1}(x, x^{i-1},\xi^{i})$ is an approximation of the function $g_{1}(x,\xi^{i})$ around the point $x^{i-1}$. Table \ref{table: SSUM} summarizes the SSUM algorithm.

\begin{table*}
\caption{The SSUM algorithm }\label{table: SSUM}
\begin{center}
\begin{tabular}{|l|}
\hline \\[0pt]
 Find a feasible point $x^{0}\in\mathcal{X}$ and set $r=0$.\\
\textbf{repeat}\\
~~~~~$r\leftarrow r+1$\\
~~~~~$\displaystyle x^{r}\leftarrow\arg\min_{x\in\mathcal{X}}~\frac{1}{r}\sum_{i=1}^{r}\left({\hat g}_{1}(x,x^{i-1},\xi^{i})+g_{2}(x,\xi^{i})\right)$\\
\textbf{until} some convergence criterion is met.\\[5pt]
\hline
\end{tabular}
\end{center}
\end{table*}
Clearly, the function $\ghat_1(x,y,\xi)$  should be related to the original function $g_1(x,\xi)$. In this paper, we assume that the approximation function $\ghat_1(x,y,\xi)$ satisfies the following conditions. \\

{\noindent\bf Assumption A:}\quad \\
Let $\cX'$ be an open set containing the set $\cX$. Suppose the approximation function $\ghat(x,y,\xi)$ satisfies the following
\begin{itemize}
\item[A1-] $\ghat_1(y,y,\xi) = g_1(y,\xi), \quad \forall\ y \in \cX, \;\forall\ \xi \in \Xi$
\item[A2-] $\ghat_1(x,y,\xi) \geq g_1 (x,\xi),\quad \forall\ x\in \cX', \;\forall\ y\in \cX,\; \forall\ \xi \in \Xi$
\item[A3-] $\ghat(x,y,\xi) \triangleq \ghat_1(x,y,\xi) + g_2(x,\xi)$ is uniformly strongly convex in $x$, i.e., for all
$\ (x,y,\xi) \in \cX\times\cX\times\Xi$,
\[
 \ghat(x+d,y,\xi) - \ghat(x,y,\xi) \geq \ghat'(x,y,\xi;d) + \frac{\gamma}{2}\|d\|^2,\  \;\forall\ d\in\mathbb{R}^n,
 \]
 where $\gamma>0$ is a constant.
\end{itemize}

The assumptions A1-A2 imply that the approximation function $\ghat_1(\cdot,y,\xi)$ should be a locally tight approximation of the original function $g_1(\cdot,\xi)$. We point out that the above assumptions can be satisfied in many cases by the right choice of the approximation function and hence are not restrictive. For example, the approximation function $\ghat_1(\cdot, y,\xi)$ can be made strongly convex easily to satisfy Assumption A3 even though the function $g_1(\cdot,y)$ itself is not even convex; see Section 3 and Section 4 for some examples.

To ensure the convergence of the SSUM algorithm, we further make the following assumptions.
\\

{\noindent\bf Assumption B:}\quad
\begin{itemize}
\item[B1-] The functions $g_1(x,\xi)$ and $\ghat_1(x,y,\xi)$ are continuous in $x$ for every fixed $y\in \cX$ and $\xi \in \Xi$
\item[B2-] The feasible set $\mathcal{X}$ is bounded
\item[B3-] The functions $g_1(\cdot,\xi)$ and $\ghat_1(\cdot,y,\xi)$, their derivatives, and their second order derivatives are uniformly bounded. In other words, there exists a constant $K>0$ such that  for all $(x,y,\xi) \in \cX\times \cX\times \Xi$ we have
\[
\begin{array}{lll}
|g_1(x,\xi)| \leq K, & \|\nabla_x g_1(x,\xi)\|\leq K, & \|\nabla_x^2 g_1(x,\xi)\|\leq K,\\
|\ghat_1(x,y,\xi)| \leq K, & \|\nabla_x \ghat_1(x,y,\xi)\|\leq K, & \|\nabla_x^2 \ghat_1(x,y,\xi)\|\leq K,\\
\end{array}
\]
\item[B4-] The function $g_2(x,\xi)$ is convex in $x$ for every fixed $\xi \in \Xi$
\item[B5-] The function $g_2(x,\xi)$ and its directional derivative are uniformly bounded. In other words, there exists $K'>0$ such that for all $(x,\xi)\in\cX\times \Xi$,  we have $|g_2(x,\xi)|\leq K'$ and
\[
|g_2'(x,\xi;d)| \leq K' \|d\|, \quad \ \forall\ d \in \mathbb{R}^n \; {\rm with}\; x+d \in \cX.
\]
\item[B6-] Let $\ghat(x,y,\xi)=\ghat_1(x,y,\xi)+g_2(x,y,\xi)$. There exists $ \bar{g} \in \mathbb{R}$ such that
\[
  |\ghat(x,y,\xi)|\leq \bar{g}, \quad \forall\;(x,y,\xi) \in \cX\times\cX\times\Xi.
\]
\end{itemize}

Notice that in the assumptions B3 and B5, the derivatives are taken with respect to the $x$ variable only. Furthermore, one can easily check that the assumption B3 is automatically satisfied if the functions $g_1(x,\xi)$ and $\ghat_1(x,y,\xi)$ are continuously second order differentiable with respect to $(x,y,\xi)$ and the set $\Xi$ is bounded; or when $g_1(x,\xi)$ and $\ghat_1(x,y,\xi)$ are continuous and second order differentiable in $(x,y)$ and $\Xi$ is finite. As will be seen later, this assumption can be easily satisfied in various practical problems. It is also worth mentioning that since the function $g_2(x,\xi)$ is assumed to be convex in $x$ in B4, its directional derivative with respect to $x$ in B5 can be written as
\begin{align}
g_2'(x,\xi;d) &= \liminf_{t\downarrow 0} \;\;\frac{g_2(x+td,\xi) - g_2(x,\xi)}{t} \nonumber\\
&= \inf_{t>0} \;\;\frac{g_2(x+td,\xi) - g_2(x,\xi)}{t} \nonumber\\
&= \lim_{t\downarrow 0} \;\;\frac{g_2(x+td,\xi) - g_2(x,\xi)}{t}. \label{eq:liminflimdirderivConvex}
\end{align}

The following theorem establishes the convergence of the SSUM algorithm.
\begin{thm}
\label{thm:Convergence}
Suppose that Assumptions A and B are satisfied. Then the iterates generated by the SSUM algorithm converge to the set of stationary points of \eqref{eq:OriginalProblem} almost surely, i.e.,
\[
\lim_{r \rightarrow \infty} d(x^r,\cX^*) = 0,
\]
where $\cX^*$ is the set of stationary points of \eqref{eq:OriginalProblem}.
\end{thm}
To facilitate the presentation of the proof, let us define the random functions
\begin{align}
f_1^r(x) &\triangleq \frac{1}{r} \sum_{i=1}^r g_1(x,\xi^i), \nonumber\\
f_2^r(x) &\triangleq \frac{1}{r} \sum_{i=1}^r g_2(x,\xi^i), \nonumber\\
\fhat_1^r(x) &\triangleq \frac{1}{r} \sum_{i=1}^r \ghat_1(x,x^{i-1},\xi^i), \nonumber\\
f^r(x) &\triangleq f_1^r(x) + f_2^r(x),\nonumber\\
\fhat^r(x) &\triangleq \fhat_1^r(x) + f_2^r(x),\nonumber
\end{align}
for $r = 1,2,\ldots$. Clearly, the above random functions depend on the realization $\xi^1,\xi^2,\ldots$ and the choice of the initial point $x^0$. Now we are ready to prove Theorem~\ref{thm:Convergence}.
\begin{proof}
First of all, since the iterates $\{x^r\}$ lie in a compact set, it suffices to show that every limit point of the iterates is a stationary point. To show this, let us consider a subsequence $\{x^{r_j}\}_{j=1}^\infty$ converging to a limit point ${\bar x}$. Note that since $\cX$ is closed, ${\bar x} \in \cX$ and therefore $\xbar$ is a feasible point. Moreover, since $|g_1(x,\xi)|<K, \; |g_2(x,\xi)|<K'$ for all $\xi \in \Xi$ (due to B3 and B5), using the strong law of large numbers \cite{FristedtProbabilityBook}, one can write
\begin{align}
\lim_{r\rightarrow \infty} \;f_1^r(x) = \mathbb{E}\left[g_1(x,\xi)\right]\triangleq f_1(x),\quad \forall\ x\in \cX, \label{eq:proof10}\\
\lim_{r\rightarrow \infty} \;f_2^r(x) = \mathbb{E}\left[g_2(x,\xi)\right]\triangleq f_2(x),\quad \forall\ x\in \cX. \label{eq:proof11}
\end{align}

\noindent Furthermore, due to the assumptions B3, B5, and \eqref{eq:liminflimdirderivConvex}, the family of functions $\{f_1^{r_j}(\cdot)\}_{j=1}^\infty$ and $\{f_2^{r_j}(\cdot)\}_{j=1}^\infty$ are equicontinuous and therefore by restricting to a subsequence, we have
\begin{align}
\lim_{j \rightarrow \infty} \;f_1^{r_j} (x^{r_j}) = \mathbb{E}_\xi \left[g_1(\xbar,\xi)\right], \label{eq:proof12}\\
\lim_{j \rightarrow \infty} \;f_2^{r_j} (x^{r_j}) = \mathbb{E}_\xi \left[g_2(\xbar,\xi)\right]. \label{eq:proof13}
\end{align}
On the other hand, $\|\nabla_x \ghat(x,y,\xi)\|<K ,\; \forall\ x,y,\xi$ due to the assumption B3 and therefore the family of functions $\{\fhat_1^r(\cdot)\}$ is equicontinuous. Moreover, they are bounded and defined over a compact set; see B2 and B4. Hence the Arzel$\rm\grave{a}$--Ascoli theorem \cite{dunford1958linear} implies that, by restricting to a subsequence, there exists a uniformly continuous function $\fhat_1(x)$ such that
\begin{align}
\lim_{j \rightarrow \infty} \; \fhat_1^{r_j} (x) = \fhat_1(x),\quad \forall\ x\in \cX, \label{eq:proof14}
\end{align}
and
\begin{align}
\lim_{j \rightarrow \infty} \; \fhat_1^{r_j} (x^{r_j}) = \fhat_1(\xbar),\quad \forall\ x\in \cX. \label{eq:proof15}
\end{align}
Furthermore, it follows from assumption A2 that
\begin{align}
\fhat_1^{r_j} (x) \geq f_1^{r_j}(x),\quad \forall\ x \in \cX'.\nonumber
\end{align}
Letting $j \rightarrow \infty$ and using \eqref{eq:proof10} and \eqref{eq:proof14}, we obtain
\begin{align}
\fhat_1 (x) \geq f_1(x),\quad \forall\ x \in \cX'.
\label{eq:proof16}
\end{align}
On the other hand, using the update rule of the SSUM algorithm, one can show the following lemma.
\begin{lem}
\label{lemma:mainlemma}
$\lim_{r \rightarrow \infty} \fhat_1^r(x^r) - f_1^r(x^r) = 0, \; {\rm almost \; surely.}$
\end{lem}

The proof of Lemma~\ref{lemma:mainlemma} is relegated to the appendix section.

Combining Lemma~\ref{lemma:mainlemma} with \eqref{eq:proof12} and \eqref{eq:proof15} yields
\begin{align}
\fhat_1(\xbar) = f_1(\xbar). \label{eq:proof17}
\end{align}
It follows from \eqref{eq:proof16} and \eqref{eq:proof17} that the function $\fhat_1(x) - f_1(x)$ takes its minimum value at the point $\xbar$ over the open set $\cX'$. Therefore, the first order optimality condition implies that
\[
\nabla \fhat_1(\xbar) - \nabla f_1(\xbar) = 0,
\]
or equivalently
\begin{align}
\nabla \fhat_1(\xbar) = \nabla f_1(\xbar). \label{eq:proof18}
\end{align}
On the other hand, using the update rule of the SSUM algorithm, we have
\begin{align}
\fhat_1^{r_j}(x^{r_j}) + f_2^{r_j}(x^{r_j}) \leq\fhat_1^{r_j}(x) + f_2^{r_j}(x),\quad \forall\ x \in \cX. \nonumber
\end{align}
Letting $j \rightarrow \infty$ and using \eqref{eq:proof13} and \eqref{eq:proof15} yield
\begin{align}
\fhat_1(\xbar) + f_2(\xbar) \leq\fhat_1(x) + f_2(x),\quad \forall\ x \in \cX. \label{eq:proof19}
\end{align}
Moreover, the directional derivative of $f_2(\cdot)$ exists due to the bounded convergence theorem \cite{FristedtProbabilityBook}. Therefore, \eqref{eq:proof19} implies that
\[
\langle \nabla \fhat_1(\xbar),d\rangle + f_2'(\xbar;d) \geq 0, \quad \forall\ d.
\]
Combining this with \eqref{eq:proof18}, we get
\[
\langle \nabla f_1(\xbar),d\rangle + f_2'(\xbar;d) \geq 0, \quad \forall\ d,
\]
or equivalently
\[
f'(\xbar;d) \geq 0, \quad \forall\ d,
\]
which means that $\xbar$ is a stationary point of $f(\cdot)$.
\end{proof}

\begin{rmk}
In Theorem~\ref{thm:Convergence}, we assume that the set $\cX$ is bounded. It is not hard to see that the result of the theorem still holds even if $\cX$ is unbounded, so long as the iterates lie in a bounded set.
\end{rmk}

\section{Applications: Transceiver Design and Online Dictionary Learning}
In this section, we describe two important applications of the SSUM approach.
\subsection{Expected Sum-Rate Maximization for Wireless Networks}
The {\it ergodic/stochastic} transceiver design problem is a long standing problem in the signal processing and communication area, and yet no efficient algorithm has been developed to date which can deliver good practical performance. In contrast, substantial progress has been made in recent years for the {\it deterministic} counterpart of this problem; see \cite{LarssonGame,RazaviyaynGame,BengtssonBook, HoniginterferencePricing,KimGiannakisBF, RazaviyaynWMMSE, RazaviyaynIABF,scutari2013decomposition,scutari2011distributed,hong2012signal}. That said, it is important to point out that most of the proposed methods require the perfect and full channel state information (CSI) of all links -- an assumption that is clearly impractical due to channel aging and channel estimation errors. More importantly, obtaining the full CSI for all links would inevitably require a prohibitively large amount of training overhead and is therefore practically infeasible.

One approach to deal with the channel aging and the full CSI problem is to use the robust optimization methodology. To date, various robust optimization algorithms have been proposed to address this issue \cite{GershmanEldarRobust, BocheRobust, EnbinRobust,TajerRobust, DavidsonShenoudaRobust,li2011coordinated}. However, these methods are typically rather complex compared to their non-robust counterparts. Moreover, they are mostly designed for the worst case scenarios and therefore, due to their nature, are suboptimal when the worst cases happen with small probability.  An alternative approach is to design the transceivers by optimizing the {\it average performance} using a stochastic optimization framework which requires only the \emph{statistical channel knowledge} rather than the full instantaneous CSI. In what follows, we propose a simple iterative algorithm for ergodic/stochastic sum rate maximization problem using the SSUM framework. Unlike the previous approach of \cite{SlockWSRMPartialCSIT} which maximizes a lower bound of the expected weighted sum rate problem, our approach directly maximizes the ergodic sum rate and is guaranteed to converge to the set of stationary points of the ergodic sum rate maximization problem. Furthermore, the proposed algorithm is computationally simple, fully distributed, and has a per-iteration complexity comparable to that of the deterministic counterpart \cite{RazaviyaynWMMSE}.

For concreteness, let us consider a $K$ cell interfering broadcast channel \cite{GroupingRazaviyayn,WMMSETSP}, where each base station $k$, $k~=~1,\ldots,K$ is equipped with $M_{k}$ antennas and serves $L_{k}$ users located in the cell $k$. We denote the $i$-th receiver in the $k$-th cell by user $i_{k}$. Let us also assume that base station $k$ wishes to transmit $d_{i_{k}}$ data streams  $\bs_{i_{k}}\in \mathbb{C}^{d_{i_{k}}}$ to the user $i_{k}$ equipped with $N_{i_{k}}$ antennas. Furthermore, we assume that the data streams are independent Gaussian random variables with $\mathbb{E}[\bs_{i_{k}}\bs_{i_{k}}^{H}]=\bI$, where $\bI$ is the identity matrix of the appropriate size. In order to keep the encoding and decoding procedure simple, we consider linear precoding strategies. In other words, the transmitted signal $\bx_{k}$ at transmitter $k$ is
\begin{align}
\bx_{k}~ =~\sum_{i=1}^{L_{k}}\bV_{i_{k}}\bs_{i_{k}},
\end{align}
where $\bV_{i_{k}}\in \mathbb{C}^{M_{k}\times d_{i_{k}}}$ is the transmit precoding matrix of user $i_k$. Due to the power consumption limitations, we assume that the average transmission power of transmitter $k$ is constrained by a budget $P_{k}$, i.e.,
\begin{align}
\sum_{i=1}^{L_{k}}\tr(\bV_{i_{k}}\bV_{i_{k}}^{H})\leq P_{k}.
\end{align}
As the transmitters use the same frequency band, the received signal is the superposition of the signals from different transmitters. Hence the received signal of user $i_k$ can be written as
\begin{align}
\by_{i_{k}}~=~\sum_{j=1}^{K} \bH_{i_{k}j}\bx_{j}+\bn_{i_{k}},
\end{align}
where $\bn_{i_{k}}$ denotes the additive white Gaussian noise with distribution $\mathcal{CN}(0,\sigma_{i_{k}}^{2}\bI)$ and $\bH_{i_{k}j}\in\mathbb{C}^{N_{i_{k}}\times M_{j}}$ is the channel matrix from transmitter $j$ to user $i_k$.
Furthermore, we assume that receiver $i_{k}$ utilizes a linear precoder $\bU_{i_{k}}\in \mathbb{C}^{N_{i_{k}}\times d_{i_{k}}}$ to estimate the data stream $\bs_{i_{k}}$, i.e.,
\begin{align}
{\hat\bs}_{i_{k}}=\bU_{i_{k}}^{H}\by_{i_{k}},
\end{align}
where ${\hat \bs}_{i_{k}}$ is the estimated signal.

Under these assumptions, it is known that the instantaneous achievable rate of user $i_k$ is given by \cite{WMMSETSP,guo2005mutual,sampath2001generalized}:
\begin{align}
R_{i_{k}}(\bU_{i_{k}},\bV,\bH)~=~\log\det(\bE_{i_{k}}^{-1}(\bV,\bU_{i_{k}}))\label{eq: Rate-function},
\end{align}
where
\begin{align}
\bE_{i_{k}}(\bV,\bU_{i_{k}}) ~\triangleq~(\bI{-}\bU_{i_{k}}^H \bH_{i_{k}k} \bV_{i_{k}})(\mathbf{I}{-}\bU_{i_{k}}^H \bH_{i_{k}k} \bV_{i_{k}})^H
~+~\nonumber\\
\sum_{(j,\ell)\neq (k,i)} \bU_{i_{k}}^H {\bH}_{i_{k}j} \bV_{\ell_{j}} \bV_{\ell_{j}}^H {\bH}_{i_{k}j}^H \bU_{i_{k}} + \sigma_{i_{k}}^2 \bU_{i_{k}}^H \bU_{i_{k}}\label{eq: MSE}
\end{align}
denotes the mean square error matrix. Furthermore, it is not hard to check that the optimum receive beamformer $\bU_{i_k}$ which maximizes \eqref{eq: Rate-function} is the MMSE receiver \cite{WMMSETSP,guo2005mutual,sampath2001generalized}:
\begin{equation} \label{eq:MMSE_Receiver}
\bU_{i_k}^{\textrm{mmse}} = \mathbf{J}_{i_k}^{-1} \bH_{i_kk}\bV_{i_k},
\end{equation}
where $\mathbf{J}_{i_k}\triangleq\sum_{j=1}^K \sum_{\ell=1}^{I_j}\bH_{i_kj}\bV_{\ell_j}\bV_{\ell_j}^H\bH_{i_kj}^H+\sigma_{i_k}^2\bI$ is the covariance matrix of the total received signal at receiver~$i_k$.

To maximize the sum of the rates of all users in the network, we need to solve
\begin{align}
\max_{\bV,\bU_{i_{k}}}~&   \sum_{k=1}^{K}\sum_{i=1}^{L_{k}}R_{i_{k}}(\bU_{i_{k}},\bV,\bH)  \nonumber\\
\st~&~\sum_{i=1}^{L_{k}}\tr(\bV_{i_{k}}\bV_{i_{k}}^{H})\leq P_{k},~\forall~k=1,\cdots,K,\nonumber
\end{align}
which requires the knowledge of all instantaneous channel matrices $\bH$ at the transmitters. Due to practical limitations, the exact channel state information (CSI) of all channels, which changes rapidly in time, is typically not available at the base stations. A more realistic assumption is to assume that an approximate CSI is know for a few links, while a statistical model of the CSI is known for the rest of the links. The latter changes more slowly in time and is easier to track.  In such situations, we are naturally led to maximize the expected sum rate of all users, where the expectation is taken over the channel statistics.

Notice that, in practice, the optimum receive beamformer in \eqref{eq:MMSE_Receiver} can be updated by measuring the received signal covariance matrix. Hence even though the complete channel knowledge is not available at the transmitters, the receive beamformers can be optimized according to the instantaneous  channel values by measuring the received signal covariance matrices. Therefore, the expected sum rate maximization problem can be written as
\begin{align}
\max_{\bV}~&\mathbb{E}_{\bH}\left\{   \sum_{k=1}^{K}\sum_{i=1}^{L_{k}}\max_{\bU_{i_{k}}}\left\{R_{i_{k}}(\bU_{i_{k}},\bV,\bH) \right\}\right\}\label{eq: ExpRate1}\\
\st~&~\sum_{i=1}^{L_{k}}\tr(\bV_{i_{k}}\bV_{i_{k}}^{H})\leq P_{k},~\forall~k=1,\cdots,K,\nonumber
\end{align}

%
To be consistent with the rest of the paper, let us rewrite \eqref{eq: ExpRate1} as a minimization problem:
\begin{align}
\min_{\bV}~&\mathbb{E}_{\bH}\{g_{1}(\bV,\bH)\}\label{eq: ExpRate}\\
\st~&~\sum_{i=1}^{L_{k}}\tr(\bV_{i_{k}}\bV_{i_{k}}^{H})\leq P_{k},~\forall~k=1,\cdots,K,\nonumber
\end{align}
where
\begin{align}
g_{1}(\bV,\bH)~=~\sum_{k=1}^{K}\sum_{i=1}^{L_{k}}\min_{\bU_{i_{k}}}\left\{-R_{i_{k}}(\bU_{i_{k}},\bV,\bH)\right\}.\label{eq: objective}
\end{align}
It can be checked that $g_{1}$ is smooth but non-convex in $\bV$ \cite{WMMSETSP}. In practice, due to other design requirements, one might be interested in adding some convex non-smooth regularizer to the above objective function. For example, the authors of \cite{hong2013joint} added a convex group sparsity promoting regularizer term to the objective for the purpose of joint base station assignment and beamforming optimization. In such a case, since the non-smooth part is convex, the SSUM algorithm is still applicable. For simplicity, we consider only the simple case of  $g_{2}\equiv 0$ in this section.

In order to utilize the SSUM algorithm, we need to find a convex tight upper-bound approximation of $g_{1}(\bV,\bH)$. To do so, let us introduce a set of variables $\bP \triangleq (\bW,\bU,\bZ)$, where $\bW_{i_{k}}\in \mathbb{C}^{d_{i_{k}}\times d_{i_{k}}}$ (with $\bW_{i_{k}}\succeq \mathbf{0}$) and $\bZ_{i_{k}}\in \mathbb{C}^{M_{k}\times d_{i_{k}}}$ for any $i=1,\cdots,L_k$ and for all  $k=1,\cdots,K$. Furthermore, define
\begin{align}
{\hat R}_{i_{k}}(\bW_{i_{k}},\bZ_{i_{k}},\bU_{i_{k}},\bV,\bH)~\triangleq~ -\log\det(\bW_{i_{k}})~+~&
\tr(\bW_{i_{k}}\bE_{i_{k}}(\bU_{i_{k}},\bV))~+~\nonumber\\
&\frac{\rho}{2}\|\bV_{i_{k}}-\bZ_{i_{k}}\|^{2}-d_{i_{k}} \label{eq:Rhat},
\end{align}
for some fixed $\rho>0$ and
\begin{align}
{\cal G}_{1}(\bV,\bP,\bH) ~\triangleq~\sum_{k=1}^{K}\sum_{i=1}^{L_{k}}{\hat R}_{i_{k}}(\bW_{i_{k}},\bZ_{i_{k}},\bU_{i_{k}},\bV,\bH).
\end{align}
Using the first order optimality condition, we can check that
\begin{align}
g_{1}(\bV,\bH)~=~\min_{\bP}~&~ {\cal G}_{1}(\bV,\bP,\bH). \nonumber
\end{align}
Now, let us define
\[
\ghat_1(\bV,\bar\bV,\bH) = {\cal G}_1(\bV , {\cal P}(\bar\bV,\bH),\bH),
\]
where
\[
{\cal P} (\bar\bV,\bH) = \arg\min_\bP \; {\cal G}_1 (\bar\bV,\bP,\bH).
\]
Clearly,  we have
\[
g_1(\bar\bV,\bH) = \min_\bP\; {\cal G}_1(\bar\bV,\bP,\bH) = {\cal G}_1(\bar\bV,{\cal P}(\bar\bV,\bH),\bH) = \ghat_1(\bar\bV,\bar\bV,\bH),
\]
and
\[
g_1(\bV,\bH) = \min_\bP\; {\cal G}_1(\bV,\bP,\bH) \leq {\cal G}_1 (\bV,{\cal P}(\bar\bV,\bH),\bH) = \ghat_1(\bV,\bar\bV,\bH).
\]
Furthermore, $\ghat_1(\bV,\bar\bV,\bH)$ is strongly convex in $\bV$ with parameter $\rho$ due to the quadratic term in \eqref{eq:Rhat}. Hence $\ghat_1(\bV,\bar\bV,\bH)$ satisfies the assumptions A1-A3. In addition, if the channels lie in a bounded subset with probability one and the noise power $\sigma_{i_k}^2$ is strictly positive for all users, then it can be checked that $g_1(\bV,\bH)$  and $\ghat_1(\bV,\bar\bV,\bH)$ satisfy the assumptions B1-B6. Consequently, we can apply the SSUM algorithm to solve \eqref{eq: ExpRate}.

Define $\bH^r$ to be the $r$-th channel realization. Let us further define
\begin{equation}
\label{eq:Pidef}
\bP^r \triangleq \arg\min_{\bP} \; {\cal G}_1(\bV^{r-1},\bP,\bH^r),
\end{equation}
where $\bV^{r-1}$ denotes the transmit beamformer at iteration $r-1$. Notice that $\bP^r$ is well defined since the optimizer of \eqref{eq:Pidef} is unique. With these definitions, the update rule of the SSUM algorithm becomes
\begin{equation}
\nonumber
\begin{split}
\bV^{r} \leftarrow \arg\min_{\bV} \;& \frac{1}{r} \sum_{i=1}^r \ghat_1(\bV,\bV^{i-1},\bH^i)\\
\st \quad & \sum_{i=1}^{L_k} \tr(\bV_{i_k}\bV_{i_k}^H) \leq P_k, \quad\forall\ k
\end{split}
\end{equation}
or equivalently
\begin{equation}
\label{eq:Vupdateproblem}
\begin{split}
\bV^{r} \leftarrow \arg\min \;& \frac{1}{r} \sum_{i=1}^r {\cal G}_1(\bV,\bP^i,\bH^i)\\
\st \quad & \sum_{i=1}^{L_k} \tr(\bV_{i_k}\bV_{i_k}^H) \leq P_k, \quad\forall\ k.
\end{split}
\end{equation}
In order to make sure that the SSUM algorithm can efficiently solve \eqref{eq: ExpRate}, we need to confirm that the update rules of the variables $\bV$ and $\bP$ can be performed in a computationally efficient manner in \eqref{eq:Pidef} and \eqref{eq:Vupdateproblem}. Checking the first order optimality condition of \eqref{eq:Pidef}, it can be shown that the updates of the variable $\bP = (\bW,\bU,\bZ)$ can be done in closed form; see Table~\ref{table: SWMMSE}. Moreover, for updating the variable $\bV$, we need to solve a simple quadratic problem in \eqref{eq:Vupdateproblem}. Using the Lagrange multipliers, the update rule of the variable $\bV$ can be performed using a one dimensional search method over the Lagrange multiplier \cite{WMMSETSP}.  Table \ref{table: SWMMSE} summarizes the SSUM algorithm applied to the expected sum rate maximization problem; we name this algorithm as stochastic weighted mean square error minimizations (stochastic WMMSE) algorithm. Notice that although in the SSUM algorithm the update of the precoder $\bV_{i_k}$ depends on all the past realizations, Table \ref{table: SWMMSE} shows that all the required information (for updating $\bV_{i_{k}}$) can be encoded into two matrices $\bA_{i_{k}}$ and $\mathbf{B}_{i_{k}}$, which are updated recursively.
\begin{rmk}
Similar to the deterministic WMMSE algorithm \cite{WMMSETSP} which works for the general $\alpha$-fairness utility functions, the Stochastic WMMSE algorithm can also be extended to maximize the expected sum of such utility functions; see \cite{WMMSETSP} for more details on the derivations of the respective update rules.
\end{rmk}

\begin{table}
\caption{SSUM algorithm applied to expected sum rate maximization}\label{table: SWMMSE}
\begin{center}
\begin{tabular}{|l|}
\hline \\[0pt]
 Initialize $\bV$ randomly such that $\sum_{i=1}^{L_{k}}\tr\left(\bV_{i_{k}}\bV_{i_{k}}^H\right) =P_k,\;\forall\; k$ and set $r=0$.\\
\textbf{repeat}\\
~~~~~$r\leftarrow r+1$\\
~~~~~Obtain the new channel estimate/realization $\bH^r$\\
~~~~~$\bU_{i_{k}} \leftarrow \left(\sum_{j=1}^{K}\sum_{l=1}^{L_{j}} \bH_{i_{k} j}^r\bV_{l_{j}} \bV_{l_{j}}^H (\bH_{i_{k} j}^r)^H+\sigma_{i_{k}}^2\bI\right)^{-1}\bH_{i_{k}k}^r\bV_{i_{k}},\; \forall\; k,~ i=1,\cdots,L_{k}$\\
~~~~~$\bW_{i_{k}} \leftarrow \left(\bI-\bU_{i_{k}}^H\bH_{i_{k}k}^r\bV_{i_{k}}\right)^{-1}, \;\forall\; k,~i=1,\cdots,L_{k}$\\
~~~~~$\bZ_{i_{k}} \leftarrow \bV_{i_{k}},\;\forall\; k,~i=1,\cdots,L_{k}$\\
~~~~~$\mathbf{A}_{i_{k}} \leftarrow \mathbf{A}_{i_{k}} + \rho\bI+ \sum_{j=1}^K\sum_{l=1}^{L_{j}}(\bH_{l_{j}k}^r)^H \bU_{l_{j}}\bW_{l_{j}}\bU_{l_{j}}^H \bH_{l_{j}k}^r,\;\forall\; k,~i=1,\cdots,L_{k} $\\
~~~~~$\mathbf{B}_{i_{k}} \leftarrow \mathbf{B}_{i_{k}} + \rho \bZ_{i_{k}} + (\bH_{i_{k}k}^r)^H \bU_{i_{k}} \bW_{i_{k}},\;\forall\; k,~i=1,\cdots,L_{k}$\\
~~~~~$\bV_{i_{k}} \leftarrow (\mathbf{A}_{i_{k}}+ \mu_k^* \bI)^{-1} \mathbf{B}_{i_{k}},\;\forall\; k,~i=1,\cdots,L_{k}$, where $\mu^{*}_{k}$ is the optimal Lagrange\\
~~~~~multiplier for the constraint $\sum_{i=1}^{L_{k}}\tr(\bV_{k}\bV_{k}^{H})\leq P_{k}$ which can be found using bisection.\\
\textbf{until} some convergence criterion is met.\\[10pt]
\hline
\end{tabular}
\end{center}
\end{table}

\subsection{Numerical Experiments}
In this section we numerically evaluate the performance of the SSUM algorithm for maximizing the expected sum-rate in a wireless network. In our simulations, we consider $K=57$ base stations each equipped with $M=4$ antennas and serve a  two antenna user in its own cell. The path loss and the power budget of the transmitters are generated using the 3GPP (TR 36.814) evaluation methodology \cite{3gpp}. We assume that partial channel state information is available for some of the links. In particular, each user estimates only its direct link, plus the interfering links whose powers are at most  $\eta$ (dB) below its direct channel power. For these estimated links, we assume a channel estimation error model in the form of $\hat{h}=h+z$, where $h$ is the actual channel; $\hat{h}$ is the estimated channel, and $z$ is the estimation error. Given a MMSE channel estimate $\hat{h}$, we can determine the distribution of ${h}$ as $\mathcal{CN} (\hat{h}, \frac{\sigma_{l}^{2}}{1+\gamma {\rm SNR}})$ where $\gamma$ is the effective signal to noise ratio (SNR) coefficient depending on the system parameters (e.g. the number of pilot symbols used for channel estimation) and $\sigma_{l}$ is the path loss. Moreover, for the channels which are not estimated, we assume the availability of estimates of the path loss $\sigma_{l}$ and use them to construct statistical models (Rayleigh {\color{black}fading} is considered on top of the path loss).

We compare the performance of four different algorithms: {\it one sample WMMSE}, {\it mean WMMSE}, {\it stochastic gradient}, and  {\it Stochastic WMMSE}. In ``one sample WMMSE" and ``mean WMMSE", we apply the WMMSE algorithm \cite{WMMSETSP} on one realization of all channels and mean channel matrices respectively. In the SG method, we apply the stochastic gradient method with diminishing step size rule to the ergodic sum rate maximization problem; see Section \ref{sec: SG}. Figure \ref{fig: Sim1} {\color{black}shows our simulation results when each user only estimates about 3\% of its channels, while the others are generated synthetically according to the channel distributions}. The expected sum rate in each iteration is approximated in this figure by a Monte-Carlo averaging over 500 independent channel realizations. As can be seen from Figure \ref{fig: Sim1}, the Stochastic WMMSE algorithm significantly outperforms the rest of the algorithms. Although the stochastic gradient algorithm with diminishing step size (of order $1/r$) is guaranteed to converge to a stationary solution, its convergence speed is sensitive to the step size selection and is usually slow. We have also experimented the SG method with different constant step sizes in our numerical simulations, but they typically led to divergence.

Figure \ref{fig: Sim2} illustrates the performance of the algorithms for $\eta =12$ whereby about 6\% of the channels are estimated.


\begin{figure}[H]
  \caption{Expected sum rate vs. iteration number. We set $\eta = 6$, $\gamma = 1$ and consequently only 3\% of the channel matrices are estimated, while the rest are generated by their path loss coefficients plus Rayleigh fading. The signal to noise ratio is set ${\rm SNR}= 15$ (dB).}\label{fig: Sim1}
  \centering
    \includegraphics[width=0.7\textwidth]{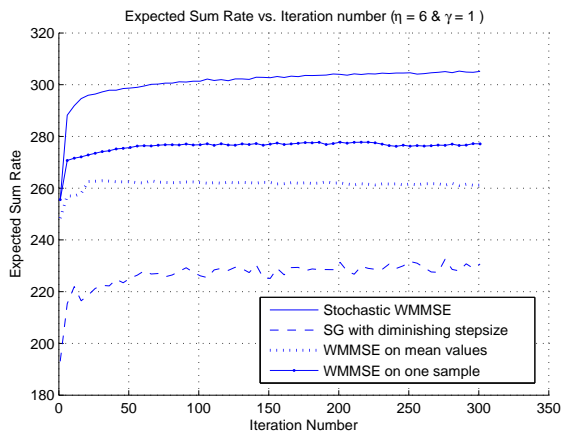}
\end{figure}

\begin{figure}[H]
  \caption{Expected sum rate vs. iteration number. We set $\eta = 12$, $\gamma = 1$ and consequently only 6\% of the channel matrices are estimated, while the rest are represented by their path loss coefficients plus Rayleigh fading. The signal to noise ratio is set ${\rm SNR}= 15$ (dB).}\label{fig: Sim2}
  \centering
    \includegraphics[width=0.7\textwidth]{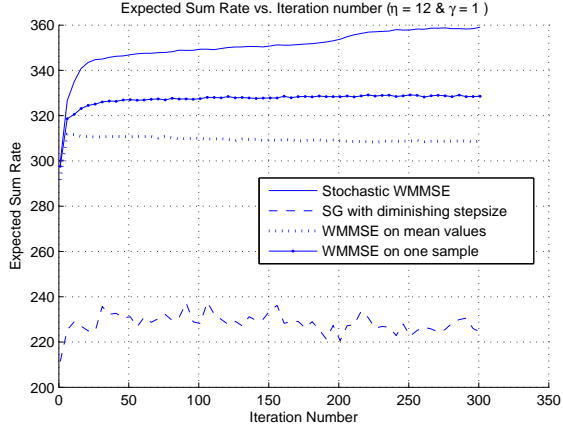}
\end{figure}

\subsection{Online Dictionary Learning}
Consider the classical dictionary learning problem: Given a random signal $y \in \mathbb{R}^n$ drawn from a distribution $P_{Y}(y)$, we are interested in finding a dictionary $D \in \mathbb{R}^{n\times k}$ so that the empirical cost function
\[
f(D) \triangleq \mathbb{E}_y \left[g(D,y)\right]
\]
is minimized over the feasible set $\mathcal{D}$; see \cite{aharon2005k,lewicki2000learning,Mairal2010SparseDL}. The loss function $g(D,y)$ measures the fitting error of the dictionary $D$ to the signal $y$. Most of the classical and modern loss functions can be represented in the form of
\begin{equation}
\label{eq:tempDL1}
g(D,y) \triangleq \min_{\alpha \in \mathcal{A}} \; h(\alpha,D,y),
\end{equation}
where $\mathcal{A} \subseteq \mathbb{R}^k$ and $h(\alpha,D,y)$ is a convex function in $\alpha$ and $D$ separately. For example, by choosing $h(\alpha,D,y) = \frac{1}{2} \|y-D\alpha\|_2^2 + \lambda \|\alpha\|_1$, we obtain the sparse dictionary learning problem; see \cite{Mairal2010SparseDL}.\\

In order to apply the SSUM framework to the online dictionary learning problem, we need to choose an appropriate approximation function $\ghat(\cdot)$. To this end, let us define
\[
\ghat(D,\bar{D},y) = h(\bar\alpha ,D,y) + \frac{\gamma}{2} \|D- \bar{D}\|_2^2,
\]
where
\[
\bar\alpha  \triangleq \arg\min_{\alpha \in \mathcal{A}} h(\alpha,\bar{D},y).
\]
Clearly, we have
\[
\ghat(\bar{D},\bar{D},y) = h(\bar\alpha , \bar{D},y) = \min_{\alpha\in \mathcal{A}} h(\alpha,\bar{D},y) = g(\bar{D},y),
\]
and
\[
\ghat(D,\bar{D},y) \geq h(\bar{\alpha},D,y) \geq g(D,y).
\]
Furthermore, if we assume that the solution of \eqref{eq:tempDL1} is unique, the function $g(\cdot)$ is smooth due to Danskin's Theorem \cite{Bertsekas_Book_Nonlinear}. Moreover, the function $\ghat(D,\bar{D},y)$ is strongly convex in $D$. Therefore, the assumptions A1-A3 are satisfied. In addition, if we assume that the feasible set $\mathcal{D}$ is bounded and the signal vector $y$ lies in a bounded set $\mathcal{Y}$, the assumptions B1-B6 are satisfied as well. Hence the SSUM algorithm is applicable to the online dictionary learning problem.
\begin{rmk}
Choosing $h(\alpha,D,y) = \frac{1}{2} \|y-D\alpha\|_2^2 + \lambda \|\alpha\|_1$ and $\gamma =0$ leads to the online sparse dictionary learning algorithm in \cite{Mairal2010SparseDL}. Notice that the authors of \cite{Mairal2010SparseDL} had to assume the uniform strong convxity of $\frac{1}{2} \|y-D\alpha\|_2^2$ for all $\alpha\in\mathcal{A}$ since they did not consider the quadratic proximal term $\gamma \|D-{\bar D}\|^2$.
\end{rmk}

\section{Stochastic (Sub-)Gradient Method and its Extensions}\label{sec: SG}
In this section, we show that the classical SG method, the incremental gradient method and the stochastic sub-gradient method are special cases of the SSUM method. We also present an extension of these classical methods using the SSUM framework.

To describe the SG method, let us consider a special (unconstrained smooth) case of the optimization problem~\eqref{eq:OriginalProblem}, where  $g_{2}\equiv 0$ and $\mathcal{X} = \mathbb{R}^n$. One of the popular algorithms for solving this problem is the stochastic gradient (also known as stochastic approximation) method. At each iteration $r$ of the stochastic gradient (SG) algorithm, a new realization $\xi^{r}$ is obtained and $x$ is updated based on the following simple rule \cite{shapiro2009lectures,robbins1951stochastic,kiefer1952stochastic,nemirovski2009robust}:
\begin{align}
x^{r}~\leftarrow~x^{r-1}-\gamma^{r}\nabla_{x}g_{1}(x^{r-1},\xi^{r}).\label{eq: SGupdate}
\end{align}
Here $\gamma^{r}$ is the step size at iteration $r$. Due to its simple update rule, the SG algorithm has been widely  used in various applications such as data classification \cite{amari1967theory,Wijnhoven2010b}, training multi-layer neural networks \cite{grippo2000,mangasarian1994,luo1991LMS,luo1994backpropagation}, the expected risk minimization \cite{bottou1998online}, solving least squares in statistics \cite{bertsekas1997leastsquares}, and distributed inference in sensor networks \cite{Bertsekas_Book_Distr,tsitsiklis1986,bertsekas1983distributed}. Also the convergence of the SG algorithm is well-studied in the literature; see, e.g., \cite{ermol1998stochastic,nemirovski2009robust,shapiro2009lectures}.

The popular incremental gradient method \cite{luo1991LMS,luo1994backpropagation,bertsekas2011incremental,bertsekas1997leastsquares,bottou1998online} can be viewed as a special case of the SG method where the set $\Xi$ is finite. In the  incremental gradient methods, a large but finite set of samples $\Xi$ is available and the objective is to minimize the empirical expectation
\begin{align}
{ \mathbb{\hat E}}\{g(x,\xi)\}~=~\frac{1}{|\Xi|}\sum_{\xi\in\Xi}g(x,\xi).\label{eq: EmpiricalExp}
\end{align}
At each iteration $r$ of the incremental gradient method (with random updating order), a new realization  $\xi^{r}\in\Xi$ is chosen randomly and uniformly, and then \eqref{eq: SGupdate} is used to update $x$. This is precisely the SG algorithm applied to the minimization of \eqref{eq: EmpiricalExp}. In contrast to the batch gradient algorithm which requires computing $\sum_{\xi\in\Xi}\nabla_{x}g(x,\xi)$, the updates of the incremental gradient algorithm are computationally cheaper, especially if $|\Xi|$ is very large.

In general, the convergence of the SG method depends on the proper choice of the step size $\gamma^{r}$. It is known that for the constant step size rule, the SG algorithm might diverge even for a convex objective function; see \cite{luo1991LMS} for an example. There are many variants of the SG algorithm with different step size rules \cite{tseng1998incremental}. In the following, we introduce a special form of the SSUM algorithm that can be interpreted as the SG algorithm with diminishing step sizes. Let us define
\begin{align}
\ghat_{1}(x,y,\xi)=g_{1}(y,\xi)+\langle\nabla g_1(y,\xi),x-y\rangle+\frac{\alpha}{2}\|x-y\|^{2}, \label{eq:SGghat}
\end{align}
where $\alpha$ is a function of $y$ and is chosen so that $\ghat_{1}(x,y,\xi)\geq g_{1}(x,\xi)$. One simple choice is $\alpha^{r} = L$, where $L$ is the Lipschitz constant of $\nabla_{x}g_{1}(x,\xi)$. Choosing  $\ghat_{1}$ in this way, the assumptions A1-A3 are clearly satisfied. Moreover, the update rule of the SSUM algorithm becomes
\begin{align}
x^{r}\;\leftarrow\arg\min_{x}~\frac{1}{r}\sum_{i=1}^{r}\ghat_{1}(x,x^{i-1},\xi^{i}).\label{eq: SSUM-SG}
\end{align}
Checking the first order optimality condition of \eqref{eq: SSUM-SG}, we obtain
\begin{align}
x^{r}~\leftarrow~\frac{1}{\sum_{i=1}^{r}\alpha^{i}}\left(\sum_{i=1}^{r}(\alpha^{i}x^{i-1}-\nabla_{x}g_{1}(x^{i-1},\xi^{i}))\right). \label{eq: SSUM-SG1}
\end{align}
Rewriting \eqref{eq: SSUM-SG1} in a recursive form yields
\begin{align}
x^{r}~\leftarrow~x^{r-1}-\frac{1}{\sum_{i=1}^{r}\alpha^{i}}\nabla_{x}g_{1}(x^{r-1},\xi^{r}),\label{eq: SSUM-SG2}
\end{align}
which can be interpreted as the stochastic gradient method \eqref{eq: SGupdate} with $\gamma^{r}=\frac{1}{\sum_{i=1}^{r}\alpha_{i}}$. Notice that the simple constant choice of $\alpha^{i}=L$ yields $\gamma^{r}=\frac{1}{rL}$, which gives the most popular diminishing step size rule of the SG method.

\begin{rmk}
When $\cX$ is bounded and using the approximation function in \eqref{eq:SGghat}, we see that the SSUM algorithm steps become
\begin{align}
z^r &\displaystyle = \frac{1}{\sum_{i=1}^{r} \alpha^i} \left(\sum_{i=1}^{r-1} \alpha^iz^{r-1} + {\alpha^r}x^{r-1} -  \nabla_x g_1(x^{r-1},\xi^r)\right), \nonumber\\
x^r &= \Pi_\mathcal{X} (z^r), \nonumber
\end{align}
where $\Pi_\mathcal{X} (\cdot)$ signifies the projection operator to the constraint set $\mathcal{X}$. Notice that this update rule is different from the classical SG method as it requires generating the auxiliary iterates $\{z^r\}$ which may not lie in the feasible set $\cX$.
\end{rmk}

It is also worth noting that in the presence of the non-smooth part of the objective function, the SSUM algorithm becomes different from the classical stochastic sub-gradient method \cite{shapiro2009lectures,robbins1951stochastic,kiefer1952stochastic,nemirovski2009robust}. To illustrate the ideas, let us consider a simple deterministic nonsmooth function $g_2(x)$ to be added to the objective function. The resulting optimization problem becomes
\[
\min_x \quad \mathbb{E}\left[g_1(x,\xi)\right] + g_2(x).
\]
Using the approximation introduced in \eqref{eq:SGghat}, the SSUM update rule can be written as
\begin{equation}
\label{eq:tempSG1}
x^r \leftarrow \arg \min_x \quad \frac{1}{r} \sum_{i=1}^r \ghat_1(x,x^{i-1},\xi^i) + g_2(x).
\end{equation}
Although this update rule is similar to the (regularized) dual averaging method \cite{nesterov2009primal,xiao2010dual} for convex problems,  its convergence is guaranteed even for the nonconvex nonsmooth objective function under the assumptions of Theorem~\ref{thm:Convergence}. Moreover, similar to the (regularized) dual averaging method, the steps of the SSUM algorithm are computationally cheap for some special nonsmooth functions. As an example, let us consider the special non-smooth function $g_2(x) \triangleq \lambda \|x\|_1$. Setting $\alpha^r = L$, the first order optimality condition of \eqref{eq:tempSG1} yields the following update rule:
\begin{equation}
\label{eq:tempSGnonsmooth}
\begin{split}
z^{r+1} & \leftarrow \frac{rz^r + x^r - \frac{1}{L}\nabla g_1 (x^r,\xi^{r+1})}{r+1},\\
x^{r+1} & \leftarrow {\rm shrink}_{\frac{\lambda}{L}} (z^{r+1}),
\end{split}
\end{equation}
where $\{z^{r+1}\}_{r=1}^{\infty}$ is an auxiliary variable sequence and ${\rm shrink}_{\tau}(z)$ is the soft shrinkage operator defined as
\begin{equation}
\nonumber
{\rm shrink}_\tau(z) = \left\{
\begin{array}{cl}
z-\tau &\quad  z\geq \tau \\
0 &\quad \tau \geq z\geq -\tau \\
z+\tau &\quad z\leq -\tau \\
\end{array}
\right.  .
\end{equation}
Notice that the algorithm obtained in \eqref{eq:tempSGnonsmooth} is different from the existing stochastic subgradient algorithm and the stochastic proximal gradient algorithm \cite{shalev2011stochastic,bertsekas2011incremental}; furthermore, if the conditions in Theorem~\ref{thm:Convergence} is satisfied, its convergence is guaranteed even for nonconvex objective functions.\\

\noindent{\bf Acknowledgment:} The authors are grateful to Maury Bramson of the University of Minnesota for valuable discussions.

\quad\\
\vspace{0.2cm}

\noindent \large{\bf Appendix} \vspace{0.3cm}

\noindent {\bf Proof of Lemma~\ref{lemma:mainlemma}}
\vspace{0.1cm}

\normalsize
\noindent The proof requires the use of quasi martingale convergence theorem  \cite{Fisk1965Quasimtgle}, much like the convergence proof of online learning algorithms \cite[Proposition 3]{Mairal2010SparseDL}. In particular, we will show that the sequence $\{\fhat^r(x^r)\}_{r=1}^{\infty}$ converges almost surely. 
Notice that
\begin{align}
&\!\!\fhat^{r+1}(x^{r+1}) - \fhat^r(x^r) \nonumber \\
&=  \fhat^{r+1}(x^{r+1}) - \fhat^{r+1}(x^r) + \fhat^{r+1}(x^r)- \fhat^r(x^r) \nonumber\\
&= \fhat^{r+1}(x^{r+1}) - \fhat^{r+1}(x^r) + \frac{1}{r+1} \sum_{i=1}^{r+1} \ghat(x^r,x^{i-1},\xi^i) - \frac{1}{r} \sum_{i=1}^r \ghat(x^r,x^{i-1},\xi^i) \nonumber\\
&= \fhat^{r+1}(x^{r+1}) - \fhat^{r+1}(x^r) - \frac{1}{r(r+1)} \sum_{i=1}^r \ghat(x^r,x^{i-1},\xi^i) + \frac{1}{r+1} \ghat(x^r,x^r,\xi^{r+1}) \nonumber\\
&= \fhat^{r+1}(x^{r+1}) - \fhat^{r+1}(x^r) - \frac{\fhat^r(x^r)}{r+1} + \frac{1}{r+1} g(x^r,\xi^{r+1})\nonumber\\
&\leq \frac{-\fhat^r(x^r) + g(x^r,\xi^{r+1})}{r+1},\nonumber
\end{align}
where the last equality is due to the assumption A1 and the inequality is due to the update rule of the SSUM algorithm. Taking the expectation with respect to the natural history yields
\begin{align}
\mathbb{E}\left[\fhat^{r+1}(x^{r+1}) - \fhat^r(x^r)\bigg|\mathcal{F}^r\right] & \leq
\mathbb{E}\left[\frac{-\fhat^r(x^r) + g(x^r,\xi^{r+1})}{r+1}\bigg|\mathcal{F}^r\right]  \nonumber\\
& = \frac{-\fhat^r(x^r)}{r+1} + \frac{f(x^r)}{r+1} \nonumber\\
& = \frac{-\fhat^r(x^r) + f^r(x^r)}{r+1} + \frac{f(x^r) - f^r(x^r)}{r+1} \label{eq:Lemma19}\\
&\leq \frac{f(x^r) - f^r(x^r)}{r+1} \label{eq:Lemma191}\\
&\leq \frac{\|f-f^r\|_\infty}{r+1}, \label{eq:Lemma20}
\end{align}
where \eqref{eq:Lemma191} is due to the assumption A2 and \eqref{eq:Lemma20} follows from the definition of $\|\cdot\|_\infty$. On the other hand, the Donsker theorem (see \cite[Lemma 7]{Mairal2010SparseDL} and \cite[Chapter 19]{VandervaartDonsker2000}) implies that there exists a constant $k$ such that
\begin{equation}
\label{eq:Lemma21}
\mathbb{E}\left[\|f-f^r\|_\infty\right] \leq \frac{k}{\sqrt{r}}.
\end{equation}
Combining \eqref{eq:Lemma20} and \eqref{eq:Lemma21} yields
\begin{equation}
\label{eq:Lemma22}
\mathbb{E} \left[\left(\mathbb{E}\left[\fhat^{r+1}(x^{r+1}) - \fhat^r(x^r)\bigg|\mathcal{F}^r\right]\right)_+\right] \leq \frac{k}{r^{3/2}},
\end{equation}
where $(a)_+ \triangleq \max\{0,a\}$ is the projection to the non-negative orthant. Summing \eqref{eq:Lemma22} over $r$, we obtain
\begin{equation}
\label{eq:Lemma23}
\sum_{r=1}^\infty\mathbb{E} \left[\left(\mathbb{E}\left[\fhat^{r+1}(x^{r+1}) - \fhat^r(x^r)\bigg|\mathcal{F}^r\right]\right)_+\right] \leq M < \infty,
\end{equation}
where $M \triangleq \sum_{r=1}^\infty \frac{k}{r^{3/2}}$. The equation~\eqref{eq:Lemma23} combined with the quasi-martingale convergence theorem (see \cite{Fisk1965Quasimtgle} and \cite[Theorem 6]{Mairal2010SparseDL}) implies that the stochastic process$\{\fhat^r(x^r) + \bar{g}\}_{r=1}^{\infty}$ is a quasi-martingale with respect to the natural history $\{\mathcal{F}^r\}_{r=1}^{\infty}$ and $\fhat^r(x^r)$ converges. Moreover, we have
\begin{equation}
\label{eq:Lemma24}
\sum_{r=1}^\infty\bigg|\mathbb{E}\left[\fhat^{r+1}(x^{r+1}) - \fhat^r(x^r)\big|\mathcal{F}^r\right]\bigg|  < \infty,\quad {\rm almost\;surely.}
\end{equation}
Next we use \eqref{eq:Lemma24} to show that  $\sum_{r=1}^\infty \frac{\fhat^r(x^r) - f^r(x^r)}{r+1} <\infty,\; {\rm almost\;surely}$. To this end, let us rewrite \eqref{eq:Lemma19} as
\begin{equation}
\label{eq:Lemma25}
\frac{\fhat^r(x^r) - f^r(x^r)}{r+1}\leq \mathbb{E}\left[-\fhat^{r+1}(x^{r+1}) + \fhat^r(x^r)\bigg|\mathcal{F}^r\right] + \frac{f(x^r) - f^r(x^r)}{r+1}.
\end{equation}
Using the fact that $\fhat^r(x^r) \geq f^r(x^r),\;\forall\ r$ and summing \eqref{eq:Lemma25} over all values of $r$, we have
\begin{equation}
\label{eq:Lemma26}
\begin{split}
0 &\leq\sum_{r=1}^\infty\frac{\fhat^r(x^r) - f^r(x^r)}{r+1}\\
& \leq \sum_{r=1}^\infty \bigg|\mathbb{E}\left[-\fhat^{r+1}(x^{r+1}) + \fhat^r(x^r)\big|\mathcal{F}^r\right]\bigg| + \sum_{r=1}^\infty\frac{\|f-f^r\|_\infty}{r+1}.
\end{split}
\end{equation}
Notice that the first term in the right hand side is finite due to \eqref{eq:Lemma24}. Hence in order to show $\sum_{r=1}^\infty\frac{\fhat^r(x^r) - f^r(x^r)}{r+1}<\infty,\;{\rm almost\;surely}$, it suffices to show that $\sum_{r=1}^\infty\frac{\|f-f^r\|_\infty}{r+1} < \infty,\;{\rm almost\;surely}$. To show this, we use the Hewitt-Savage zero-one law; see \cite[Theorem 11.3]{HewittSavage01law} and \cite[Chapter 12, Theorem 19]{FristedtProbabilityBook}. Let us define the event
$$\mathcal{A} \triangleq \left\{(\xi^1,\xi^2,\ldots) \mid \sum_{r=1}^\infty \frac{\|f^r-f\|_\infty}{r+1}<\infty\right\}.$$
It can be checked that the event $\mathcal{A}$ is permutable, i.e., any finite permutation of each element of $\mathcal{A}$ is inside $\mathcal{A}$; see \cite[Theorem 11.3]{HewittSavage01law} and \cite[Chapter 12, Theorem 19]{FristedtProbabilityBook}. Therefore, due to the Hewitt-Savage zero-one law \cite{HewittSavage01law}, probability of the event $\mathcal{A}$ is either zero or one. On the other hand, it follows from \eqref{eq:Lemma21} that there exists $M'>0$ such that

\begin{equation}
\label{eq:Lemma27}
\mathbb{E}\left[\sum_{r=1}^\infty \frac{\|f^r-f\|_\infty}{r+1}\right] \leq M' < \infty.
\end{equation}
Using Markov's inequality, \eqref{eq:Lemma27} implies that
\[
Pr\left(\sum_{r=1}^\infty \frac{\|f^r-f\|_\infty}{r+1} > 2M'\right) \leq \frac{1}{2}.
\]
Hence combining this result with the result of the Hewitt-Savage zero-one law, we obtain $Pr(\mathcal{A}) =1$; or equivalently
\begin{equation}
\label{eq:Lemma28}
\sum_{r=1}^\infty \frac{\|f^r-f\|_\infty}{r+1} < \infty, \quad {\rm almost \; surely}.
\end{equation}
As a result of \eqref{eq:Lemma26} and \eqref{eq:Lemma28}, we have
\begin{equation}
\label{eq:Lemma29}
0 \leq \sum_{r=1}^\infty \frac{\fhat^r(x^r) - f^r(x^r)}{r+1} < \infty, \quad {\rm almost \; surely}.
\end{equation}
On the other hand, it follows from the  triangle inequality that
\begin{eqnarray}
&&\!\!\!\!\!\!\bigg|\fhat^{r+1}(x^{r+1}) - f^{r+1}(x^{r+1}) - \fhat^r(x^r) + f^r(x^r)\bigg|\nonumber\\
&&\leq \bigg| \fhat^{r+1}(x^{r+1})- \fhat^r(x^r)\bigg| + \bigg| f^{r+1}(x^{r+1})- f^r(x^r)\bigg|
\label{eq:Lemma30}
\end{eqnarray}
and
\begin{align}
&\bigg| \fhat^{r+1}(x^{r+1})- \fhat^r(x^r)\bigg| \nonumber\\
&\leq \!\bigg| \fhat^{r+1}(x^{r+1})- \fhat^{r+1}(x^r)\bigg| + \bigg| \fhat^{r+1}(x^r)\!-\! \fhat^r(x^r)\bigg| \nonumber\\
& \leq \!\kappa\|x^{r+1} - x^r\|\! + \!\left|\frac{1}{r+1} \sum_{i=1}^{r+1} \ghat(x^r,x^{i-1},\xi^i)\! -\! \frac{1}{r} \sum_{i=1}^{r} \ghat(x^r,x^{i-1},\xi^i) \right| \label{eq:Lemma301} \\
& \leq  \!\kappa\|x^{r+1} - x^r\|\! +\!\left| \frac{1}{r(r+1)} \sum_{i=1}^{r} \ghat(x^r,x^{i-1},\xi^i)\! +\! \frac{\ghat(x^r,x^r,\xi^{r+1})}{r+1}\right|  \nonumber\\
& \leq \!\kappa\|x^{r+1} - x^r\| + \frac{2\bar{g}}{r+1} \label{eq:Lemma302} \\
& = \!\mathcal{O}\left (\frac{1}{r}\right),\label{eq:Lemma31}
\end{align}
where \eqref{eq:Lemma301} is due to the assumption B3 (with $\kappa=(K+K') $); \eqref{eq:Lemma302} follows from the assumption B6, and \eqref{eq:Lemma31} will be shown in Lemma~\ref{lemma:O1overR}. Similarly, one can show that

\begin{equation}
\label{eq:Lemma32}
|f^{r+1}(x^{r+1}) - f^r(x^r)| = \mathcal{O}\left(\frac{1}{r}\right).
\end{equation}
It follows from \eqref{eq:Lemma30}, \eqref{eq:Lemma31}, and \eqref{eq:Lemma32} that
\begin{equation}
\label{eq:Lemma33}
\bigg|\fhat^{r+1}(x^{r+1}) - f^{r+1}(x^{r+1}) - \fhat^r(x^r) + f^r(x^r)\bigg| = \mathcal{O}\left(\frac{1}{r}\right).
\end{equation}
Let us fix a random realization $\{\xi^r\}_{r=1}^\infty$ in the set of probability one for which \eqref{eq:Lemma29} and \eqref{eq:Lemma33} hold. Define
\[
\alpha^r \triangleq \fhat^r(x^r) - f^r(x^r).
\]
Clearly, $\alpha^r \geq 0$ and $\sum_r \frac{\alpha^r}{r}<\infty$ due to \eqref{eq:Lemma29}. Moreover, it follows from \eqref{eq:Lemma33} that $|\alpha^{r+1}-\alpha^r| < \frac{\tau}{r}$ for some constant $\tau>0$. Hence Lemma~\ref{lemma:sequencelimitzero} implies that
\[
\lim_{r \rightarrow \infty} \; \alpha^r = 0,
\]
which is the desired result.

\begin{lem}
\label{lemma:O1overR}
$\|x^{r+1} - x^r\| = \mathcal{O}(\frac{1}{r}).$
\end{lem}
\begin{proof}
The proof of this lemma is similar to the proof of \cite[Lemma 1]{Mairal2010SparseDL}; see also \cite[Proposition 4.32]{BonnansShapiroBook2000}.
First of all, since $x^r$ is the minimizer of $\fhat^r(\cdot)$, the first order optimality condition implies
\[
\fhat^r(x^r;d) \geq 0, \quad \forall\ d\in \mathbb{R}^n.
\]
Hence, it follows from the assumption A3 that
\begin{equation}
\label{eq:lemma51}
\fhat^r(x^{r+1}) - \fhat^r(x^r) \geq \frac{\gamma}{2} \|x^{r+1} - x^r\|^2.
\end{equation}
On the other hand,
\begin{align}
\fhat^r(x^{r+1}) - \fhat^r(x^r) &\leq \fhat^r(x^{r+1}) - \fhat^{r+1}(x^{r+1}) + \fhat^{r+1}(x^r) - \fhat^r(x^r) \label{eq:Lemma52}\\
&\leq \frac{1}{r(r+1)} \sum_{i=1}^r |\ghat(x^{r+1},x^{i-1},\xi^i) - \ghat(x^r,x^{i-1},\xi^i)| \nonumber\\
&\quad  + \frac{1}{r+1}|\ghat(x^{r+1},x^r,\xi^{r+1}) - \ghat(x^r,x^r,\xi^{r+1})| \nonumber\\
&\leq \frac{\theta}{r+1} \|x^{r+1} - x^r\| \label{eq:lemma53},
\end{align}
where \eqref{eq:Lemma52} follows from the fact that $x^{r+1}$ is the minimizer of $\fhat^{r+1}(\cdot)$, the second inequality is due to the definitions of $\fhat^r$ and $\fhat^{r+1}$, while \eqref{eq:lemma53} is the result of the assumptions B3 and B5. Combining \eqref{eq:lemma51} and \eqref{eq:lemma53} yields the desired result.
\end{proof}
\begin{lem}
\label{lemma:sequencelimitzero}
Assume $\alpha^r > 0$ and $\sum_{r=1}^\infty \frac{\alpha^r}{r} <\infty$. Furthermore, suppose that $|\alpha^{r+1}-\alpha^r| \le {\tau}/{r}$ for all $r$. Then $\lim_{r\rightarrow \infty} \alpha^r = \infty$.
\end{lem}
\begin{proof}
Since  $\sum_{r=1}^{\infty} \frac{\alpha^r}{r} < \infty$, we have $\liminf_{r \rightarrow \infty} \alpha^r = 0$. Now, we prove the result using contradiction. Assume the contrary so that
\begin{equation}
\label{eq:Lemma61}
\limsup_{r \rightarrow \infty} \alpha^r >\epsilon,
\end{equation}
for some $\epsilon >0$. Hence there should exist subsequences $\{m_j\}$ and $\{n_j\}$ with $m_j\leq n_j<m_{j+1},\forall\ j$ so that
\begin{align}
\frac{\epsilon}{3} < \alpha^r \quad\quad &m_j \leq r <n_j, \label{eq:Lemma62}\\
\alpha^r \leq \frac{\epsilon}{3} \quad\quad & n_j \leq r < m_{j+1}. \label{eq:Lemma63}
\end{align}
On the other hand, since $\sum_{r=1}^{\infty} \frac{\alpha^r}{r} < \infty$, there exists an index $\bar{r}$ such that
\begin{equation}
\label{eq:Lemma64}
\sum_{r = \bar{r}}^{\infty} \frac{\alpha^r}{r} < \frac{\epsilon^2}{9\tau}.
\end{equation}
Therefore, for every $r_0 \geq \bar{r}$ with $m_j \leq r_0 \leq n_j-1$, we have
\begin{align}
|\alpha^{n_j} - \alpha^{r_0}| &\leq \sum_{r = r_0}^{n_j-1}|\alpha^{r+1} - \alpha^r| \nonumber\\
&\leq \sum_{r = r_0}^{n_j-1} \frac{\tau}{r} \label{eq:Lemma66}\\
&\leq \frac{3}{\epsilon} \sum_{r = r_0}^{n_j-1} \frac{\tau}{r} \alpha^r \label{eq:Lemma67} \\
&\leq \frac{3\tau \epsilon^2}{9\epsilon \tau} = \frac{\epsilon}{3},\label{eq:Lemma68}
\end{align}
where the equation \eqref{eq:Lemma67} follows from \eqref{eq:Lemma62}, and \eqref{eq:Lemma68} is the direct consequence of \eqref{eq:Lemma64}. Hence the triangle inequality implies
\[
\alpha^{r_0} \leq \alpha^{n_j} + |\alpha^{n_j} - \alpha^{r_0}| \leq \frac{\epsilon}{3} + \frac{\epsilon}{3} = \frac{2\epsilon}{3},
\]
for any $r_0 \geq \bar{r}$, which contradicts \eqref{eq:Lemma61}, implying that
$$
\limsup_{r\rightarrow \infty} \alpha^r = 0.
$$
\end{proof}

\bibliographystyle{spmpsci}

\bibliography{biblio}
\end{document}